\documentclass[11pt]{article}
\usepackage{amssymb}
\newtheorem{precor}{{\bf Corollary}}

\newtheorem{precon}{{\bf Conjecture}}

\newtheorem{prealphcon}{{\bf Conjecture}}

\newtheorem{predefin}{{\bf Definition}}

\newtheorem{preexm}{{\bf Example}}

\newtheorem{preappl}{{\bf Application}}

\newtheorem{prelem}{{\bf Lemma}}

\newtheorem{preproof}{{\bf Proof.\ }}

\newenvironment{proof}[1]{\begin{preproof}{\rm
               #1}\hfill{$\blacksquare$}}{\end{preproof}}
\newtheorem{pretheorem}{{\bf Theorem}}

\newenvironment{theorem}{\begin{pretheorem}{\hspace{-0.5
               em}{\bf.\ }}}{\end{pretheorem}}
\newtheorem{prealphtheorem}{{\bf Theorem}}

\newtheorem{prealphlem}{{\bf Lemma}}

\newtheorem{prepro}{{\bf Proposition}}

\newtheorem{preprb}{{\bf Problem}}

\newtheorem{prerem}{{\bf Remark}}

\newtheorem{preapp}{{\bf Application}}

\newtheorem{prequ}{{\bf Question}}

\def\conct[#1,#2]{\mbox {${#1} \leftrightarrow {#2}$}}
\def\dconct[#1,#2]{\mbox {${#1} \rightarrow {#2}$}}
\def\deg[#1,#2]{\mbox {$d_{_{#1}}(#2)$}}
\def\mindeg[#1]{\mbox {$\delta_{_{#1}}$}}
\def\maxdeg[#1]{\mbox {$\Delta_{_{#1}}$}}
\def\outdeg[#1,#2]{\mbox {$d_{_{#1}}^{^+}(#2)$}}
\def\minoutdeg[#1]{\mbox {$\delta_{_{#1}}^{^+}$}}
\def\maxoutdeg[#1]{\mbox {$\Delta_{_{#1}}^{^+}$}}
\def\indeg[#1,#2]{\mbox {$d_{_{#1}}^{^-}(#2)$}}
\def\minindeg[#1]{\mbox {$\delta_{_{#1}}^{^-}$}}
\def\maxindeg[#1]{\mbox {$\Delta_{_{#1}}^{^-}$}}

\def\dre[#1,#2,#3]{\mbox {${\cal E}^{^{#3}}(#1,#2)$}}
\def\var[#1,#2]{\mbox {${\rm Var}_{_{#1}}(#2)$}}
\def\ls[#1]{\mbox {$\xi^{^{#1}}$}}
\def\hom[#1,#2]{\mbox {${\rm Hom}({#1},{#2})$}}
\def\onvhom[#1,#2]{\mbox {${\rm Hom^{v}}(#1,#2)$}}
\def\onehom[#1,#2]{\mbox {${\rm Hom^{e}}(#1,#2)$}}
\def\core[#1]{\mbox {$#1^{^{\bullet}}$}}
\def\cay[#1,#2]{\mbox {${\rm Cay}({#1},{#2})$}}
\def\sch[#1,#2,#3]{\mbox {${\rm Sch}({#1},{#2},{#3})$}}
\def\cays[#1,#2]{\mbox {${\rm Cay_{s}}({#1},{#2})$}}
\def\dirc[#1]{\mbox {$\stackrel{\rightarrow}{C}_{_{#1}}$}}
\def\cycl[#1]{\mbox {${\bf Z}_{_{#1}}$}}

\begin{document}

\begin{center}
{\Large \bf About Fall Colorings of Graphs}\\
\vspace{0.3 cm}
{\bf Saeed Shaebani}\\
{\it Department of Mathematical Sciences}\\
{\it Institute for Advanced Studies in Basic Sciences (IASBS)}\\
{\it P.O. Box {\rm 45195-1159}, Zanjan, Iran}\\
{\tt s\_shaebani@iasbs.ac.ir}\\ \ \\
\end{center}
\begin{abstract}
\noindent A fall $k$-coloring of a graph $G$ is a proper
$k$-coloring of $G$ such that each vertex of $G$ sees all $k$
colors on its closed neighborhood. In this paper, we answer some
questions of \cite{dun} about some relations between fall
colorings and some other types of graph colorings. \noindent
\end{abstract}
\section{Introduction}
All graphs considered in this paper are finite and simple
(undirected, loopless and without multiple edges). Let $G=(V,E)$
be a graph and $k\in \mathbb{N}$ and $[k]:=\{i|\ i\in
\mathbb{N},\ 1\leq i\leq k \}$. A $k$-coloring (proper k-coloring)
of $G$ is a function $f:V\rightarrow [k]$ such that for each
$1\leq i\leq k$, $f^{-1}(i)$ is an independent set. We say that
$G$ is $k$-colorable whenever $G$ admits a $k$-coloring $f$, in
this case, we denote $f^{-1}(i)$ by $V_{i}$ and call each $1\leq
i\leq k$, a color (of $f$) and each $V_{i}$, a color class (of
$f$). The minimum integer $k$ for which $G$ has a $k$-coloring,
is called the chromatic number of G and is denoted by $\chi(G)$.

Let $G$ be a graph, $f$ be a $k$-coloring of $G$ and $v$ be a
vertex of $G$. The vertex $v$ is called colorful ( or
color-dominating or $b$-dominating) if each color $1\leq i\leq k$
appears on the closed neighborhood of $v$ (\ $f(N[v])=[k]$ ). The
$k$-coloring $f$ is said to be a fall $k$-coloring (of $G$) if
each vertex of $G$ is colorful. There are graphs $G$ for which
$G$ has no fall $k$-coloring for any positive integer $k$. For
example, $C_{5}$ (a cycle with 5 vertices) and graphs with at
least one edge and one isolated vertex, have not any fall
$k$-colorings for any positive integer $k$. The notation ${\rm
Fall}(G)$ stands for the set of all positive integers $k$ for
which $G$ has a fall $k$-coloring. Whenever ${\rm
Fall}(G)\neq\emptyset$, we call $\min({\rm Fall}(G))$ and
$\max({\rm Fall}(G))$, fall chromatic number of $G$ and fall
achromatic number of $G$ and denote them by $\chi_{f}(G)$ and
$\psi_{f}(G)$, respectively. Every fall $k$-coloring of a graph
$G$ is a $k$-coloring, hence, for every graph $G$ with ${\rm
Fall}(G)\neq\emptyset$, $\chi(G)\leq\chi_{f}(G)\leq\psi_{f}(G)$.

Let $G$ be a graph, $k\in \mathbb{N}$ and $f$ be a $k$-coloring
of $G$. The coloring $f$ is said to be a colorful $k$-coloring of
$G$ if each color class contains at least one colorful vertex.
The maximum integer $k$ for which $G$ has a colorful
$k$-coloring, is called the $b$-chromatic number of $G$ and is
denoted by $\phi(G)$ (or $b(G)$ or $\chi_{b}(G)$). Every fall
$k$-coloring of $G$ is obviously a colorful $k$-coloring of $G$
and therefore, for every graph G with ${\rm
Fall}(G)\neq\emptyset$,
$\chi(G)\leq\chi_{f}(G)\leq\psi_{f}(G)\leq\phi(G)$.

Assume that $G$ is a graph, $k\in \mathbb{N}$ and $f$ is a
$k$-coloring of $G$ and $v$ be a vertex of $G$. The vertex $v$ is
called a Grundy vertex (with respect to $f$) if each color $1\leq
i<f(v)$ appears on the neighborhood of $v$. The k-coloring $f$ is
called a Grundy $k$-coloring (of $G$) if each color class of $G$
is nonempty and each vertex of $G$ is a Grundy vertex. The maximum
integer $k$ for which $G$ has a Grundy $k$-coloring, is called
the Grundy chromatic number of $G$ and is denoted by $\Gamma(G)$.
Also, the k-coloring $f$ is said to be a partial Grundy
$k$-coloring (of $G$) if each color class contains at least one
Grundy vertex. The maximum integer $k$ for which $G$ has a
partial Grundy $k$-coloring, is called the partial Grundy
chromatic number of $G$ and is denoted by $\partial\Gamma(G)$.
Every Grundy $k$-coloring of $G$ is a partial Grundy $k$-coloring
of $G$ and every colorful $k$-coloring of $G$ is a partial Grundy
$k$-coloring of $G$. Also, every fall $k$-coloring of $G$ is
obviously a Grundy $k$-coloring of $G$ and therefore, for every
graph G with ${\rm Fall}(G)\neq\emptyset$,

$\chi(G)\leq\chi_{f}(G)\leq\psi_{f}(G)\leq \left\{\begin{array}{c}
  \phi(G) \\
  \Gamma(G)
\end{array}\right.\leq \partial\Gamma(G)$.

Let $G$ be a graph and $k\in \mathbb{N}$ and $f$ be a $k$-coloring
of $G$. The k-coloring $f$ is said to be a complete $k$-coloring
(of $G$) if there is an edge between any two distinct color
classes. The maximum integer $k$ for which $G$ has a complete
$k$-coloring, is called the achromatic number of $G$ and is
denoted by $\psi(G)$. Every partial Grundy $k$-coloring of $G$ is
obviously a complete $k$-coloring of G and therefore,

$\chi(G)\leq\chi_{f}(G)\leq\psi_{f}(G)\leq \left\{\begin{array}{c}
  \phi(G) \\
  \Gamma(G)
\end{array}\right.\leq\partial\Gamma(G)\leq\psi(G)$.

The terminology fall coloring was firstly introduced in 2000 in
\cite{dun} and has received attention recently, see
\cite{MR2096633},\cite{MR2193924},\cite{dun},\cite{las}. The
colorful coloring of graphs was introduced in 1999 in \cite{irv}
with the terminology $b$-coloring. The concept of Grundy number of
graphs was introduced in 1979 in \cite{chr}. Also, achromaric
number of graphs was introduced in 1970 in \cite{har}.

Let $n \in \mathbb{N}$ and for each $1\leq i\leq n$, $G_{i}$ be a
graph. The graph with vertex set $\bigcup_{i=1}^{n}(\{i\}\times
V(G_{i}))$ and edge set $$[\bigcup_{i=1}^{n}\{\{(i,x),(i,y)\}|
\{x,y\}\in E(G_{i})\}]\bigcup[\bigcup_{1\leq i<j\leq n}\{
\{(i,a),(j,b)\}|a\in V(G_{i}),b\in V(G_{j})\}]$$ is called the
join graph of $G_{1},...,G_{n}$ and is denoted by
$\bigvee_{i=1}^{n}G_{i}$.

Cockayne and Hedetniemi proved in 1976 in \cite{coc} (but not
with the terminology "fall coloring") that if $G$ has a fall
$k$-coloring and $H$ has a fall $l$-coloring for positive
integers $k$ and $l$, then, $G\bigvee H$ has a fall
$(k+l)$-coloring.

\begin{theorem}{
Let $n\in \mathbb{N}\setminus\{1\}$ and for each $1\leq i\leq n$,
$G_{i}$ be a graph. Then,

${\rm Fall}(\bigvee_{i=1}^{n}G_{i})\neq\emptyset$ iff for each
$1\leq i\leq n$, ${\rm Fall}(G_{i})\neq\emptyset$.}
\end{theorem}
\begin{proof}{
First suppose that ${\rm
Fall}(\bigvee_{i=1}^{n}G_{i})\neq\emptyset$. Consider a fall
$k$-coloring $f$ of $\bigvee_{i=1}^{n}G_{i}$. The colors appear
on $\{i\}\times V(G_{i})$ form a fall $|f(\{i\}\times
V(G_{i}))|$-coloring of $G_{i}$ ( Let $S$ be the set of colors
appear on $\{i\}\times V(G_{i})$ and $\alpha,\beta\in S$ and
$\alpha\neq\beta$ and $x\in \{i\}\times V(G_{i})$ and
$f(x)=\alpha$ and suppose that none of the neighbors of $x$ in
$\{i\}\times V(G_{i})$ have the color $\beta$. Since $f$ is a fall
$k$-coloring of $\bigvee_{i=1}^{n}G_{i}$, there exists a vertex
$y\in V(\bigvee_{i=1}^{n}G_{i})$ such that $\{x,y\}\in
E(\bigvee_{i=1}^{n}G_{i})$ and $f(y)=\beta$. On the other hand,
since $\beta\in S$, there exists a vertex $z$ in $\{i\}\times
V(G_{i})$ such that $f(z)=\beta$. Since $x$ and $y$ are adjacent,
$z$ and $y$ are adjacent, too. Also, $f(z)=f(y)=\beta$, which is a
contradiction. Therefore, $S$ forms a $|S|$-coloring of the
induced subgraph of $\bigvee_{i=1}^{n}G_{i}$ on $\{i\}\times
V(G_{i})$ and also for $G_{i}$.) and therefore, ${\rm
Fall}(G_{i})\neq\emptyset$. Conversely, suppose that for each
$1\leq i\leq n$, $k_{i}\in {\rm Fall}(G_{i})$. For each $1\leq
i\leq n$, construct a fall $k_{i}$-coloring of the induced
subgraph of $\bigvee_{i=1}^{n}G_{i}$ on $\{i\}\times V(G_{i})$
with the color set
$\{(\sum_{j=1}^{i-1}k_{j})+1,(\sum_{j=1}^{i-1}k_{j})+2,...,(\sum_{j=1}^{i-1}k_{j})+k_{i}\}$.
This forms a fall $(\sum_{i=1}^{n}k_{i})$-coloring of
$\bigvee_{i=1}^{n}G_{i}$ and therefore, ${\rm
Fall}(\bigvee_{i=1}^{n}G_{i})\neq\emptyset$. }\end{proof}

The proof of the following obvious theorem has omitted for the
sake of brevity.

\begin{theorem}{
\label{6parts} Let $n\in \mathbb{N}\setminus\{1\}$ and for each
$1\leq i\leq n$, $G_{i}$ be a graph. Then,

1)If for each $1\leq i\leq n$, ${\rm Fall}(G_{i})\neq\emptyset$,
then, ${\rm Fall}(\bigvee_{i=1}^{n}G_{i})=\sum_{i=1}^{n}{\rm
Fall}(G_{i}):=\{a_{1}+...+a_{n}|\ a_{1}\in {\rm
Fall}(G_{1}),...,\ a_{n}\in {\rm Fall}(G_{n})\}$ and
$\chi_{f}(\bigvee_{i=1}^{n}G_{i})=\sum_{i=1}^{n}\chi_{f}(G_{i})$
and
$\psi_{f}(\bigvee_{i=1}^{n}G_{i})=\sum_{i=1}^{n}\psi_{f}(G_{i})$.

2) $\chi(\bigvee_{i=1}^{n}G_{i})=\sum_{i=1}^{n}\chi(G_{i})$.

3) $\phi(\bigvee_{i=1}^{n}G_{i})=\sum_{i=1}^{n}\phi(G_{i})$.

4) $\Gamma(\bigvee_{i=1}^{n}G_{i})=\sum_{i=1}^{n}\Gamma(G_{i})$.

5)
$\partial\Gamma(\bigvee_{i=1}^{n}G_{i})=\sum_{i=1}^{n}\partial\Gamma(G_{i})$.

6) $\psi(\bigvee_{i=1}^{n}G_{i})=\sum_{i=1}^{n}\psi(G_{i})$.}
\end{theorem}

In \cite{dun}, Dunbar, et al. asked the following questions.
\\

1*) whether or not there exists a graph $G$ with ${\rm
Fall}(G)\neq\emptyset$ which satisfies
$\chi_{f}(G)-\chi(G)\geq3$? They noticed that
$\chi_{f}(C_{4}\square C_{5})=4$ and $\chi(C_{4}\square
C_{5})=3$, also, $\chi_{f}(C_{5}\square C_{5})=5$ and
$\chi(C_{5}\square C_{5})=3$.
\\

2*) Can $\chi_{f}(G)-\chi(G)$ be arbitrarily large?
\\

3*) Does there exist a graph $G$ with ${\rm
Fall(G)}\neq\emptyset$ which satisfies
$\chi(G)<\chi_{f}(G)<\psi_{f}(G)<\phi(G)<\partial\Gamma(G)<\psi(G)$?
\\

Since $\chi_{f}(C_{4}\square C_{5})=4$ and $\chi(C_{4}\square
C_{5})=3$, Theorem \ref{6parts} implies that For each $n\in
\mathbb{N}$, $\chi_{f}(\bigvee_{i=1}^{n}(C_{4}\square
C_{5}))-\chi(\bigvee_{i=1}^{n}(C_{4}\square C_{5}))=4n-3n=n$ and
this gives an affirmative answer to the problems 1* and 2*.

Also, the Theorem \ref{6parts} and the following theorem, give an
affirmative answer to all 3 questions immediately.

\begin{theorem}{
\label{thm} For each integer $\varepsilon>0$, there exists a
graph $G$ with $Fall(G)\neq\emptyset$ which the minimum of
$\chi_{f}(G)-\chi(G)$,  $\psi_{f}(G)-\chi_{f}(G)$,
$(\delta(G)+1)-\psi_{f}(G)$, $\Gamma(G)-\psi_{f}(G)$,
$\phi(G)-\psi_{f}(G)$, $(\Delta(G)+1)-\partial\Gamma(G)$,
 $\psi(G)-\partial\Gamma(G)$,
 $\partial\Gamma(G)-\Gamma(G)$ is greater
than $\varepsilon$.}
\end{theorem}

\begin{proof}{

Let $\varepsilon>2$ be an arbitrary integer and let's follow the
following steps.

Step1) Set $G_{1}:=\bigvee_{i=1}^{\varepsilon+1}(C_{4}\square
C_{5})$. As stated above,
$\chi_{f}(G_{1})-\chi(G_{1})=\varepsilon+1$.
\\

Step2) Set $G_{2}:=K_{\varepsilon+3,\varepsilon+3}- ({\rm \ an\
arbitrary}\ 1-factor\ )$. One can easily observe that
$\psi_{f}(G_{2})-\chi_{f}(G_{2})=(\varepsilon+3)-2=\varepsilon+1$.
\\

Step3) Set $G_{3}:=K_{(\varepsilon+2,\varepsilon+2)}$. Then,
$(\delta(G_{3})+1)-\psi_{f}(G_{3})=((\varepsilon+2)+1)-2=\varepsilon+1$.
\\

Step4) Let $P_{\varepsilon+3}$ be a path with $\varepsilon+3$
vertices. Add $\varepsilon+2$ pendant vertices to each of its
vertices and denote the new graph by $G_{4}$. It is readily seen
that
$\phi(G_{4})-\psi_{f}(G_{4})\geq(\varepsilon+3)-2=\varepsilon+1$.
\\

Step5) Let $T(1)$ be the tree with only one vertex and for each
$k\geq1$, $T(k+1)$ be the graph obtained by adding a new pendant
vertex to each vertex of $T(k)$. $T(\varepsilon+3)$ is a tree
which its Grundy number is $\varepsilon+3$ and
$\psi_{f}(T(\varepsilon+3))\leq \delta(T(\varepsilon+3))+1\leq2$.
Hence, if we set $G_{5}:=T(\varepsilon+3)$, then,
$\Gamma(G_{5})-\psi_{f}(G_{5})=\varepsilon+1$.
\\

Step6) Let $G_{6}$ be the graph obtained by adding $i-2$ pendant
vertices to each vertex $v_{i}(3\leq i\leq\varepsilon+5)$ of the
path $v_{1}v_{2}\ldots v_{\varepsilon+5}$. Obviously,
$\partial\Gamma(G_{6})\geq\varepsilon+5$ and
$\Gamma(G_{6})\leq4$. So,
$\partial\Gamma(G_{6})-\Gamma(G_{6})\geq\varepsilon+1$.
\\

Step7) Set $G_{7}:=K_{\varepsilon+2,\varepsilon+2}$.
$(\Delta(G_{7})+1)-\partial\Gamma(G_{7})=(\varepsilon+3)-2=\varepsilon+1$.
\\

Step8) Set $G_{8}:=P_{\frac{(\varepsilon+4)(\varepsilon+3)}{2}}$.
Obviously, $\psi(G_{8})\geq\varepsilon+4$ and
$\partial\Gamma(G_{8})\leq\triangle(G_{8})+1\leq3$. Hence,
$\psi(G_{8})-\partial\Gamma(G_{8})\geq\varepsilon+1$.
\\

Step9) Set $G:=\bigvee_{i=1}^{8}G_{i}$. For each $1\leq i\leq 8$,
${\rm Fall}(G_{i})\neq\emptyset$. Hence, by Theorem \ref{6parts},
 the fact that
$\delta(\bigvee_{i=1}^{8}G_{i})\geq\sum_{i=1}^{8}\delta(G_{i})$,
and
$\Delta(\bigvee_{i=1}^{8}G_{i})\geq\sum_{i=1}^{8}\Delta(G_{i})$,
$G_{9}$ is a suitable graph for this theorem and also for
questions 1*, 2* and 3*.

}\end{proof}

\end{document}